\documentclass{amsart}
\usepackage{parskip}
\usepackage{amssymb}
\usepackage{amsmath}
\usepackage{amsthm}
\usepackage{amsbsy}
\usepackage{amscd}
\usepackage{bm}
\usepackage{graphicx}
\usepackage{hyperref}
\usepackage[bottom=1.7cm, right=2.5cm]{geometry}

\makeatletter
\g@addto@macro\bfseries{\boldmath}
\makeatother

\newtheorem{proposition}{Proposition}
\newtheorem{theorem}[proposition]{Theorem}
\newtheorem{definition}[proposition]{Definition}
\newtheorem{lemma}[proposition]{Lemma}

\newtheorem{corollary}[proposition]{Corollary}
\newtheorem{remark}[proposition]{Remark}

\newcommand{\SSS}{\mathbb{S}}
\newcommand{\R}{\mathbb{R}}
\newcommand{\CP}{\mathbb{CP}}
\newcommand{\D}{\mathbb{D}}
\newcommand{\B}{\mathbb{B}}

\newcommand{\N}{\mathbb{N}}

\newcommand{\Leg}{\mathcal{L}eg}

\newcommand{\Cont}{\mathcal{C}ont}
\newcommand{\DCont}{\mathcal{D}\mathcal{C}ont}
\newcommand{\FCont}{\mathcal{F}\mathcal{C}ont}

\makeatletter
\newcommand*{\defeq}{\mathrel{\vcenter{\baselineskip0.5ex \lineskiplimit0pt
                     \hbox{\scriptsize.}\hbox{\scriptsize.}}}%
                     =}
\makeatother

\makeatletter
\def\blfootnote{\xdef\@thefnmark{}\@footnotetext}
\makeatother

\begin{document}
\title{A simple construction of positive loops of Legendrians}

\author{Dishant Pancholi}
\address{Chennai Mathematical Institute, H1 SIPCOT IT Park, Siruseri, Kelambakkam 603103, India.}
\email{dishant@cmi.ac.in}

\author{Jos\'e Luis P\'erez}
\address{Instituto de Ciencias Matem\'aticas CSIC-UAM-UC3M-UCM, C. Nicol\'as Cabrera, 13-15, 28049 Madrid, Spain}
\email{joseluis.perez@icmat.es}

\author{Francisco Presas}
\address{Instituto de Ciencias Matem\'aticas CSIC-UAM-UC3M-UCM, C. Nicol\'as Cabrera, 13-15, 28049 Madrid, Spain}
\email{fpresas@icmat.es}

\date{\today}
\subjclass{Primary: 53D10; Secondary: 57R17}
\begin{abstract}
We construct positive loops of Legendrian submanifolds in several instances. In particular, we partially recover G. Liu's result stating that any loose Legendrian admits a positive loop, under some mild topological assumptions on the Legendrian. Moreover, we show contractibility of the constructed loops under an extra topological assumption.
\end{abstract}
\maketitle
\section{Introduction}

\subsection{Motivation}\label{Motivation}

Consider a ($2n+1$)--dimensional co-oriented contact manifold $(M,\xi)$. Y. Eliashberg and L. Polterovich \cite{EP} introduced the notion of \emph{non-negative contact isotopy}, they showed that it induces a relation on the identity component of the group of contactomorphisms, $\Cont_0(M,\xi)$. This relation, which was also studied by Bhupal \cite{Bu}, is naturally reflexive and transitive but not necessarily anti-symmetric. We say that $\Cont_0(M,\xi)$ is strongly orderable if the relation is also anti-symmetric, i.e. it is a partial order on $\Cont_0(M,\xi)$. Analogously, we can define a relation in the universal cover $\widetilde{\Cont}_0(M,\xi)$ that again may fail to be anti-symmetric. We say that $\widetilde{\Cont}_0(M,\xi)$ is orderable if the relation defines a partial order on $\widetilde{\Cont}_0(M,\xi)$. Contact topology has embraced the study of this relation during the last few years \cite{AFM, EKP, EP, Gi, We}.

There is a relative version of the construction. Let $L$ be a Legendrian submanifold in a contact manifold ($M,\xi$) and denote by $\Leg(L)$ the space of all Legendrian submanifolds which are Legendrian isotopic to $L$. Non-negative Legendrian isotopies also define a relation on $\Leg(L)$ an we say that $\Leg(L)$ is orderable if this relation is anti-symmetric (respectively, the universal cover $\widetilde{\Leg}(L)$ is orderable if the analogous relation is anti-symmetric).

The existence of these partial orders can be checked in terms of the non-existence of positive loops of contactomorphisms (resp. Legendrians). The study of orderability for Legendrians and of the existence of positive (contractible) loops has been an active research area in contact topology. Thus, for instance, V. Colin, E. Ferrand and P. Pushkar \cite{CFP} studied the non--existence of positive loops of Legendrian submanifolds in $\SSS T^{*}M$ where the universal cover of $M$ is the $n$--dimensional real space. In the field of Lorentzian geometry, V. Chernov and S. Nemirovsky \cite{CN1, CN2, CN3} apply this topic to the study of causality in globally hyperbolic spacetimes. The orderability property of Legendrians gives rise to the existence of bi--invariant integer--valued metrics in the space of Legendrians \cite{CS}.

Recently, G. Liu \cite{Li1, Li2} has announced the existence of (contractible) positive loops for loose Legendrian submanifolds. The goal of this note is to offer a shorter proof of G. Liu's result under some extra assumptions.

\subsection{Statement of the results}

Consider a ($2n+1$)-dimensional manifold $M$ endowed with a contact structure $\xi$. An $n$-dimensional embedded submanifold $L^n\subset M^{2n+1}$ is called \textit{Legendrian} if its tangent space at each point is contained in the contact distribution.

The key remark of this article is the following

\begin{theorem}\label{thm:product}
Let $(M, \ker(\alpha))$ be a contact manifold and fix a positive constant $\varepsilon>0$. Consider the contact manifold $(M \times \D^2_{\varepsilon}(r, \theta), \ker(\alpha + r^2 d \theta))$. Any closed Legendrian submanifold in $M\times \D^2_{\varepsilon}$  admits a positive loop of Legendrians.
\end{theorem}

A loop of Legendrian submanifolds is called positive if the generating Hamiltonian of the loop is positive. The proof is extremely simple and the core of these notes is devoted to extract some corollaries of the result. The most important one is the next

\begin{theorem}\label{thm:loose}
Let $n\geq 2$. Fix a loose Legendrian submanifold $L^n$ in a contact manifold $(M^{2n+1}, \xi)$. Assume that its bundle $T^{*}L\oplus\R$ has two independent sections. Then $L$ admits a positive loop of Legendrians.
\end{theorem}

A Legendrian submanifold is \textit{loose} if there is a special chart in an open neighborhood of a domain of the Legendrian (see Definitions \ref{def:loose_chart},  \ref{def:loose}). Murphy proved that loose Legendrians satisfy an $h$--principle \cite{Mu}. Note that this definition assumes for $2n+1 \geq 5$.

In $3$--dimensional contact topology, there is an analogous older notion \cite{EF}. A Legendrian knot in a contact $3$--fold whose complement is overtwisted is called loose. They also satisfy an $h$--principle.

If $2n+1 \geq 5$, any Legendrian submanifold whose complement is overtwisted is loose. This is a consequence of the parametric and relative nature of the $h$--principle for overtwisted contact structures (see \cite{BEM}).

For didactical reasons, we will first prove the following particular case of Theorem \ref{thm:loose}.

\begin{theorem} \label{thm:main}
Let $n\geq 1$. Assume that a Legendrian submanifold $L^n$ in a contact manifold $(M^{2n+1}, \xi)$ satisfies that the bundle $T^{*}L\oplus\R$ has two independent sections. If $M \backslash L$ is overtwisted, then $L$ admits a positive loop of Legendrians.
\end{theorem}

We remark that this result covers the $3$--dimensional situation that is not included in Theorem \ref{thm:loose}.

Realize that the hypothesis of $T^{*}L\oplus\R$ having two independent sections is pretty mild. Elaborating on the orientable Legendrians case, some sufficient conditions for it to be satisfied are:
\begin{itemize}
    \item $\chi(L)=0$. This, in particular, covers odd dimensional Legendrians.
    \item $w_2(L)=0$. Since, this implies that $w_2(T^{*}L\oplus\R)=0$ and by the definition of this obstruction class in the even dimensional case, the vanishing of the class implies the existence of two independent sections. In particular, this covers even dimensional Legendrians with even Euler class.
    \item Any Legendrian submanifold whose tangent bundle is trivialized by direct sum with $\R$. This obviously covers all the spheres.
\end{itemize}

There are simple examples of manifolds not satisfying that property. For instance, $L=\CP^2$ is a manifold whose $1$--jet bundle $T^{*}\CP^2\oplus\R$ does not admit two independent sections.

%Note that any legendrian whose complementary is overtwisted is loose. This is a consequence of the parametric and relative nature of the $h$--principle for overtwisted structures (see \cite{BEM}).

Let us move now to the study of positive contractible loops. We prove the following key remark.

\begin{theorem}\label{thm:R4}
Let $(M,\xi=\ker(\alpha))$ be a contact manifold and on the product $M\times D_{\varepsilon}^4(r_1,\theta_1,r_2,\theta_2)$ define the contact form $\widetilde{\alpha}=\alpha+r_1^2d\theta_1+r_2^2d\theta_2$. Define the domain
$$M^{+}=\{(p,r_1,\theta_1,r_2,\theta_2)\in M\times D_{\varepsilon}^4 \text{ such that } 0<r_1<r_2\}.$$
 Any Legendrian embedding
$L\hookrightarrow  M^{+}\subset M\times D_{\varepsilon}^4$ admits a contractible positive loop of Legendrians on $M\times D_{\varepsilon}^4$.
\end{theorem}

This statement implies

\begin{corollary}\label{cor:1}
Let $n\geq 3$. Fix a loose Legendrian submanifold $L^n$ in a contact manifold $(M^{2n+1}, \xi)$. Assume that its bundle $T^{*}L\oplus\R$ has four independent sections. Then, $L$ admits a contractible positive loop of Legendrians.
\end{corollary}

Again, the hypothesis can be easily checked. We cover, for instance, Legendrian spheres of dimension $n\geq 3$. Let us consider two more corollaries from Theorem \ref{thm:R4}.

\begin{corollary}\label{cor:2}
If $L\subset(\R^{2n+1}, \xi_{std})$ with $n\geq2$, then $L$ admits a contractible positive loop.
\end{corollary}

Realize that this statement can be proven by using that $\SSS^{2n+1}$ admits a contractible positive loop \cite{EKP}, placing $\R^{2n+1}\subset\SSS^{2n+1}$ and making sure that the restrictions of the contact isotopies to the Legendrian submanifold do not cross $\infty\in\SSS^{2n+1}$. This can be done by a genericity argument whenever $n\geq 2$. However, the proof presented in this note is more elementary.

\begin{corollary}\label{cor:3}
Let $\R^{2n+1}$ be the Euclidean space equipped with the overtwisted at infinity contact structure $\xi$. If $L\subset(\R^{2n+1}, \xi)$ and $n> 2$ then $L$ admits a contractible positive loop.
\end{corollary}

For the precise notion of overtwisted at infinity see Definition \ref{def:OTatInfty}.

\subsection{Acknowledgements}

We are grateful to R. Casals, V. Colin, V. Ginzburg, G. Liu and A. del Pino, for useful discussions. The first author is grateful to the Marie Curie Research programme ``Indo European collaboration on moduli spaces'' that allowed him to visit ICMAT during the development of this project. The first author is also thankful to ICTP, Trieste (Italy), where part of this work was carried out. Second and third authors are supported by the Spanish National Research Projects SEV-2015-0554 and MTM2016-79400-P.

\section{Preliminaries}\label{preliminares}

Consider a ($2n+1$)-dimensional manifold $M$ endowed with a contact structure $\xi$. An embedding  of an $n$-dimensional manifold $\phi:L^n\hookrightarrow M^{2n+1}$ is called \textit{Legendrian}\footnote{We will work along the paper with parametrized Legendrians. This is done in order to ease the notation. } if its differential $D\phi:TL\rightarrow\phi^{*}(TM)$ satisfies that $D\phi(TL)\subset\phi^{*}\xi$. %The Reeb vector field associated to the contact form $\alpha$ will be denoted by $R_{\alpha}$ or simply by $R$.

A vector field $X$ on $M$ is called a \textit{contact vector field} if its flow preserves the contact structure $\xi=\ker(\alpha)$. That is, $L_{X} \alpha = g \alpha.$ Fixing a contact form $\alpha$ there exists a bijection between the space of contact vector fields and the space of smooth functions. Given a contact vector field $X$, the function $H\defeq\alpha(X)\in\mathcal{C}^{\infty}(M)$ which satisfies the equations:

\begin{enumerate}
  \item $i_X \alpha = H$
  \item $ i_{X}  d \alpha = (d_R H) \alpha  - d H$
 \end{enumerate}

is called the \textit{associated Hamiltonian}. Conversely, given a function $H\in\mathcal{C}^{\infty}(M)$ there exists a unique contact vector field $X_H$ verifying the equations above.

A diffeomorphism $\psi$ of $(M,\xi)$ is a \textit{contactomorphism} if $\psi_*(\xi) = \xi$ or, equivalently, $\psi^* \alpha = g \alpha$ for some everywhere positive function $g$ on $M$. An \textit{isotopy of contactomorphisms} is a smooth diffeotopy $\psi_t: M \rightarrow M$ generated by a $1$--parametric family of contact vector fields $X_t$, with $t\in[0,1]$. We say that the isotopy is a \textit{loop of contactomorphisms} if $\psi_0=\psi_1=Id$.

Let us remark that the above bijection implies that the isotopy is completely characterized by a $1$--parametric family of Hamiltonians $H_t: M \rightarrow \mathbb{R}$. Hence, we can make the following definition:

\begin{definition}
An isotopy of contactomorphisms $\psi_t$ is \textit{non-negative} if its associated family of Hamiltonians $H_t$ is non-negative, i.e. $H_t(p) \geq 0$ for all $p$ in $M$ and for all $t$ in $[0,1]$. If the inequality is strict, the isotopy is called \textit{positive}. Analogously we can define\textit{ positive} and \textit{ non-negative loops of contactomorphisms}.
\end{definition}

Let us point out that when we have a loop of contactomorphisms we can choose the parameter to be defined as $t\in\SSS^1$ i.e. the Hamiltonian can be chosen to satisfy $H_t(p)=H_{t+1}(p)$. The above definitions can be adapted to this situation.

\begin{definition}
An \textit{isotopy of Legendrian submanifolds} is a smooth $1$--parametric family $\phi_t :L\rightarrow M$ of Legendrian embeddings with $t \in I=[0,1]$.  That is, a smooth map $\phi: L \times I \rightarrow M$ such that $\left.\phi\right|_{L \times \{t\}}$ is a Legendrian embedding for all $t$. By a \textit{\textit{loop of Legendrians based at}} $\Lambda$ we mean an isotopy of Legendrians such that $\phi_0(L) = \phi_1(L)=\Lambda$ as submanifolds of $(M, \xi)$.

% An \textit{isotopy of Legendrian submanifolds} is a smooth $1$--parametric family $L_t \subset M$ of Legendrians with $t \in [0,1].$  That is, a smooth map $\phi: L \times I \rightarrow M$ such that $\phi(L \times \{t\})=\phi_t(L)$ is an embedded Legendrian submanifold of $M$ for all $t$. We will always assume that $\phi_0=Id$. By a \textit{\textit{loop of Legendrians based at}} $L$ we mean an isotopy of Legendrians such that $L_0 = L_1 = L$ as submanifolds of $(M, \xi)$.
\end{definition}

A basic fact about Legendrian isotopies is the next Legendrian isotopy extension theorem:

\begin{theorem}[see, e.g., {\cite[Thm. 2.6.2]{Ge}}]
Let $\phi_t$ be a given isotopy of a closed Legendrian, then we can extend the isotopy by a family $\psi_t$ of contactomorphisms satisfying $\psi_t\circ\phi_0 = \phi_t$ and $\psi_0 = Id$ .
\end{theorem}

%We will always assume that the map $\phi : L  \times I \rightarrow M$ defined above is the identity when restricted to $L \times \{0\} = L_0.$  With this convention in mind, we have the following:

%\begin{lemma}
% Any smooth Legendrian isotopy can be generated by a contact isotopy $\psi_t : M \rightarrow M$ such that $\phi_t(L_0) = L_t.$
%\end{lemma}

We are ready to introduce the concept of positive isotopy of Legendrians.

\begin{definition}
 Let us consider the contact structure $\xi$ induced by a contact form $\alpha$. An isotopy $\phi_t$ of Legendrians is called \textit{non-negative} (resp. \textit{positive}) if $\alpha\left(\frac{\partial\phi_t}{\partial t}(p)\right)\geq0 \text{ (resp. } \alpha \left(\frac{\partial\phi_t}{\partial t}(p)\right)>0)$ for all $p\in L$ and for all $t$.
\end{definition}

Clearly, these notions are independent of the chosen parametrization. This is due to the fact that, given a different parametrization $\widetilde{\phi}:L\times I\rightarrow M$, the difference of the vector fields $\frac{\partial\phi_t}{\partial t}$ and $\frac{\partial\widetilde{\phi}_t}{\partial t}$ lies in the tangent space to the Legendrian submanifold $\phi_t(L)$.

According to the previous definition, a loop of Legendrians is called \textit{non-negative} (resp. \textit{positive}) provided the isotopy generating the loop is non-negative (resp. positive). Notice that to have a positive loop of Legendrians is much weaker than to have a positive loop of contactomorphisms. Any extension of a positive Legendrian loop needs be neither a loop of contactomorphisms nor positive. However, we can easily arrange the extension of a positive (resp. non-negative) loop of Legendrians to be positive (resp. non negative). This fact will be used afterwards.

\begin{definition}
A loop of Legendrians $\phi_t$ is contractible if there exists a homotopy of loops of Legendrians $\phi_{t,s}$ such that $\phi_{t,1}=\phi_t$, $\phi_{t,0}=\phi_{0,1}$ and $\phi_{0,s}=\phi_{0,1}$.
\end{definition}

\begin{remark}
The existence of a positive loop of a specific Legendrian $L$ implies that the space $\Leg (L)$ is not orderable. Equivalently, the existence of contractible positive loop of a specific Legendrian $L$ implies that the space $\widetilde{\Leg} (L)$ is not orderable.
\end{remark}

\subsection{Operations with loops of Legendrians}
\mbox{}\\
We can define three important operations on the space of loops of Legendrians: concatenation, composition and conjugation.

\begin{description}
  \item[2.1.1] \textbf{\underline{Concatenation}:}
  Let $\{\phi^1_t,\cdots, \phi^k_t\}$ be $k$ loops of Legendrians with fixed base point $L \subset M$. Note that the reparametrization of the loops given by

  \begin{eqnarray*}
  \widetilde{\phi}^1_t&=&\phi^1_t,\\
  \widetilde{\phi}^2_t&=&\phi^2_t\circ(\phi^2_0)^{-1}\circ\phi^1_1,\\ \widetilde{\phi}^3_t&=&\phi^3_t\circ(\phi^3_0)^{-1}\circ\phi^2_1\circ(\phi^2_0)^{-1}\circ\phi^1_1,\\
   \vdots\hspace{1.4mm} &   \vdots & \hspace{2.5cm}\vdots \\
  \widetilde{\phi}^k_t&=&\phi^k_t\circ(\phi^k_0)^{-1}\circ\phi^{k-1}_1\circ(\phi^{k-1}_0)^{-1}\circ\dots\circ\phi_1^1
  \end{eqnarray*}

  satisfies $\widetilde{\phi}_1^j=\widetilde{\phi}_0^{j+1}$ for $1\leq j\leq k-1$. Thus, we assume this property in the family of loops without loss of generality.
  Consider their associated Hamiltonians $\{H^1_t,\cdots, H^k_t\}$. The \textit{concatenation} operation $\phi^1_t  \circledcirc \cdots \circledcirc \phi^k_t$ is defined in the loop space of $\Leg(L)$ as the usual concatenation of loops:

% \[  \phi^1_t \circledcirc \cdots \circledcirc \phi^k_t   =
% \left\{ \begin{array}{ll}
%  \phi^1_{kt}  &       \mbox{if $t \in [0,\frac{1}{k}]$};\\
%  \phi^2_{kt-1} &  \mbox{if $t \in [\frac{1}{k}, \frac{2}{k}]$}; \\
%    \cdot & \cdot \\
%  \cdot & \cdot \\
%    \cdot & \cdot \\
%  \phi^{k-1}_{kt- k +2} &   \mbox{if $t \in [1 - \frac{2}{k}, 1 - \frac{1}{k}]  $} \\
%   \phi^{k}_{kt -k +1} &   \mbox{if $t \in [1 - \frac{1}{k}, 1]$}
%   \end{array} \right.. \]

 \[  \phi^1_t \circledcirc \cdots \circledcirc \phi^k_t   =
  \left\{ \begin{array}{ll}
  \phi^1_{kt}  &       \mbox{if $t \in [0,\frac{1}{k}]$};\\
  \phi^2_{kt-1} &  \mbox{if $t \in [\frac{1}{k}, \frac{2}{k}]$}; \\
  \hspace{3mm}\vdots & \hspace{4mm}\vdots \\
  \phi^{k-1}_{kt- k +2}  &   \mbox{if $t \in [1 - \frac{2}{k}, 1 - \frac{1}{k}]  $} \\
  \phi^{k}_{kt -k +1} &   \mbox{if $t \in [1 - \frac{1}{k}, 1]$}
  \end{array} \right.. \]

Fix extensions $\{\psi^1_t,\cdots, \psi^k_t\}$ and associated Hamiltonians $\{H^1_t,\cdots, H^k_t\}$, then the generating Hamiltonian of the concatenation is $H(\phi^1_t \circledcirc \cdot \circledcirc \phi^k_t)$ is given by

 \[ H(\phi^1_t \circledcirc \cdots \circledcirc \phi^k_t)   = \left\{ \begin{array}{ll}
  k H^1(\cdot,kt)  &       \mbox{if $t \in [0,\frac{1}{k}]$};\\
  k H^2(\cdot,kt-1) &  \mbox{if $t \in [\frac{1}{k}, \frac{2}{k}]$} \\
 \hspace{8mm}\vdots & \hspace{5mm}\vdots \\
  k H^{k-1}(\cdot,kt- k +2) &   \mbox{if $t \in [1 - \frac{2}{k}, 1 - \frac{1}{k}]  $} \\
   k H^{k}(\cdot,kt -k +1) &   \mbox{if $t \in [1 - \frac{1}{k}, 1]$}
   \end{array} \right. \]

  \item[2.1.2] \textbf{\underline{Composition}:}

  Let $\psi_t^1$ and $\psi_t^2$ be two extensions of two loops of Legendrians with common base point $L$ embedded in $M$ and let $H_t^1$ and $H_t^2$ be their associated Hamiltonians. Realize that the composition of the loops $\psi_t^1\circ \psi_t^2$ defines a loop of Legendrians given by $\phi_t=\left.(\psi_t^1\circ \psi_t^2)\right._{|\psi^1_0(L)}$. In addition, if ${\psi_t^1}^* \alpha = e^{f_t} \alpha$ then the associated Hamiltonian $H(\psi_t^1\circ \psi_t^2)$ for the composition is given by

 $$ H(\psi_t^1\circ \psi_t^2)(p,t) = H_t^1(p,t) + e^{-f_{t}} H_t^2(\psi_t^{-1}(p),t).  $$

  Let us remark that this operation depends on the choice of extensions and is not canonically defined in the loop space of $\Leg(L)$.

  \item[2.1.3] \textbf{\underline{Conjugation}:}

  Finally, let $\phi_t$ be a loop of Legendrians based at $L$ and let $\Psi$ be a contactomorphism, then $\Psi\circ\phi_t$ is a loop of Legendrians of $\Psi(\phi_0(L))$. Now, consider the extension $\psi_t$ of $\phi_t$ with the associated Hamiltonian $H_t$. If ${\Psi}^* \alpha = e^{f} \alpha$, then the contact isotopy $\Psi \circ \psi_t \circ \Psi^{-1}$ is an extension of the loop of Legendrians $\Psi\circ\phi_t$ and is generated by the Hamiltonian

 $$ H(\Psi \circ \phi_t \circ \Psi^{-1}) (p,t) = e^{-f}H_t(\Psi^{-1}(p),t).  $$

Let us remark that if $\Psi$ preserves $\phi_0(L)$, then the conjugated loop is still a loop based at $\phi_0(L)$. Also, the conjugation of a positive (resp. non-negative) loop is positive (resp. non-negative). This shows that the property of having a positive loop is independent of the chosen Legendrian within the isotopy class of Legendrians.

\end{description}

\subsection{Formal Legendrians and formal contact structures}

Now denote by $\Cont(M)$ the \textit{space of co-oriented contact structures} in $M$ and consider the set $\DCont(M)=\{(\xi, \alpha,J:\xi\rightarrow\xi)\}$ where $\xi\in\Cont(M)$, $\alpha$ is an associated contact form and $J$ is an almost-complex structure compatible with $(\xi, d\alpha)$. This set is known as the \textit{space of decorated contact structures}. Notice that the forgetful map $f:\DCont(M)\rightarrow\Cont(M)$ has contractible fibers. Therefore it induces a homotopy equivalence and thus it has a homotopy inverse $\iota:\Cont(M)\rightarrow\DCont(M)$.

Finally, we define the \textit{space of formal contact structures} of $M$ as the set of pairs $\FCont(M)=\{(\xi, J)\}$, where $\xi$ is a co-oriented distribution and $J:\xi\to\xi$ is an almost-complex structure. Two contact structures $\xi_1$ and $\xi_2$ are \textit{formally equivalent} if there exists a family of formal contact structures $\{(\xi_t, J_t)\}$ that connects them. Composing $\iota$ with the projection map $\pi:\DCont\to\FCont$, we get a natural map

\[j:\Cont(N)\hookrightarrow\FCont(N).\]

There is a natural inclusion $i$ given by:

\begin{align}\label{form:natincl}
i:\FCont(N)&\hookrightarrow\FCont(N\times\R^2)\\
(\xi, J)&\mapsto\left(\xi\oplus\R^2,
\left(
\begin{array}{c|c}
J & 0 \\ \hline
0 & i
\end{array}
\right)\right)\nonumber.
\end{align}

\begin{lemma}\label{lem:top_1}
 If $N$ is an open manifold, then the inclusion map (\ref{form:natincl}) induces an isomorphism
\[i_{0}: \pi_0(\FCont(N)) \rightarrow \pi_0(\FCont(N \times \mathbb{R}^2)).\]
\end{lemma}

\begin{proof}
Notice that a formal contact structure on $N$ is a reduction of the structure group to $1 \times U(n-1)$. Hence, considering a formal contact structure on $N$ is equivalent to having a section of the associated bundle $SO(2n-1)/U(n-1)$. Analogously, having a formal contact structure on $N \times \mathbb{R}^2$ is equivalent to choosing a section of the associated $(SO(2n+1)/U(n))$--bundle.

 The homotopy groups $\pi_k$ of the spaces $SO(2n-1)/U(n-1)$ and $SO(2n+1)/U(n)$ are isomorphic whenever $k<2n-1$ \cite[Lemma 8.1.2]{Ge}. Observe that, as $N$ is an open manifold, we are able to get a CW-decomposition of $N$ with no cells of dimension $2n-1$. Hence $i_{*}$ is an isomorphism.
\end{proof}

Let us remark that if $N$ is closed, the same argument provides the surjectivity of $i_{0}$.

\subsection{Looseness and Overtwistedness}

\begin{definition}\label{defn:OT3}
A contact structure $\xi$ on $M^3$ is called \textit{overtwisted} if there exists an embedded $2$--disk $\D^2\subset M$ such that $\partial\D^2\sqcup\{0\}$ is tangent to the contact distribution while the rest of the disk is transverse to $\xi$. If $\xi$ is not overtwisted, it is called \textit{tight}.
\end{definition}

We define the {\em standard overtwisted contact form} in $\R^3(z,r,\theta)$ to be $\alpha_{ot}=cos(r)dz+r\cdot sin(r)d\theta$. It is overtwisted since the embedding $e:\D_{\pi}^2\hookrightarrow\R^3$, $e(r,\theta)=(0,r,\theta)$ is overtwisted.

Overtwisted contact structures are important because they form a subclass of $\Cont(M)$ satisfying a complete $h$--principle \cite{El1, BEM} in all dimensions. Definition \ref{defn:OT3} gives us the notion of an overtwisted contact structure in dimension $3$. Let us generalize this concept.

There exists a sequence of positive constants $R(n)$ in $\R^+$, whose value is computed in \cite{CMP} that provide the following

\begin{definition}
Let $(M,\xi)$ be a contact manifold of dimension ${2n+1}>3$. $(M,\xi)$ is called \textit{overtwisted} if there exists a contact embedding $\phi_{ot}:(\B^{3}_{2\pi}\times \B_{R(n)}^{2n-2}, \ker(\alpha_{ot}+\lambda_{std})) \hookrightarrow (M,\xi)$, where $\lambda_{std}$ is the standard Liouville form on $\B_{R}^{2n-2}$. Otherwise, it is called \textit{tight}.
\end{definition}

We say that a formal contact structure is overtwisted if it is genuine in some open set $B$ and is overtwisted on $B$.

Fix a closed set $A\subset M$ and a contact structure $\xi_A$ on a germ of neighborhood of $A$. Denote by $\Cont_{ot}(M,A, \xi_A)$ the space of contact structures that are overtwisted on $M \setminus A$ and coincide with $\xi_A$ on a small neighborhood of $A$. Equivalently, define $\FCont_{ot}(M,A,\xi_A)$ to be the space of overtwisted formal contact structures that agree with $\xi_A$ on $U_A$. Finally, denote by $j$ the inclusion map $j:\Cont_{ot}(M,A,\xi_A)\to\FCont_{ot}(M,A,\xi_A)$.

\begin{theorem}[{\cite[Thm. 1.2]{BEM}}]\label{thm:BEM}
If $M\setminus A$ is connected, then the inclusion map $j$ induces an isomorphism

\[j_{0}: \pi_0(\Cont_{ot}(M,A,\xi_A)) \rightarrow \pi_0(\FCont_{ot}(M,A,\xi_A)).\]

\end{theorem}

In particular, on any closed manifold $M$, any almost contact structure is homotopic to an overtwisted contact structure which is unique up to isotopy.

%Y. Eliashberg \cite{El2} gave a complete classification of contact structures on $\R^3$. It turns out that there is four types of contact structures: overtwisted, tight, overtwisted at infinity and tight at infinity. %We have already defined overtwisted and tight contact structures. Let us define the other notions.

For the open case we have
\begin{definition}\label{def:OTatInfty}
The contact manifold $(M,\xi)$ is called \textit{overtwisted at infinity} if for any compact subset $K\subset M$, each noncompact connected component of the contact manifold $(M\setminus K,\xi)$ is overtwisted. %Otherwise, it is called \textit{tight at infinity}.
\end{definition}

Eliashberg \cite{El2} proved that two contact structures on $\R^3_{ot}$ at infinity are contactomorphic. This result can be extended, without changes in the argument, to general open manifolds of arbitrary dimension. Concretely,

%he stated that two formally equivalent contact manifolds overtwisted at infinity are contactomorphic. This result can be extended to higher dimension. Concretely,

\begin{lemma}\label{lem:swindling}
Let $M$ be an open manifold and let $(M,\xi_0)$ and $(M,\xi_1)$ be two contact structures overtwisted at infinity such that $\xi_0$ and $\xi_1$ are formally equivalent. Then, there exists a diffeomorphism $\Psi:M\to M$ such that $\Psi_{*}\xi_0=\xi_1$.
\end{lemma}
The proof is left to the reader. It follows, verbatim, \cite{El2}.

To prove Theorem \ref{thm:main}, we will use the above Lemma together with Theorem \ref{thm:product}. Hence we will need to find a contact manifold $N$ such that $N\times\D^2$ is overtwisted at infinity. The next {\em folklore} result shows that it suffices $N$ to be an overtwisted manifold.

\begin{proposition}\label{prop:OTInfty}
Let $(N, \ker(\alpha))$ be an overtwisted contact manifold. Then $(N \times \R^2, \ker(\alpha+r^2 d \theta))$ is overtwisted at infinity.
\end{proposition}

Notice that the proposition does not hold in dimension $1$ since there is no notion of overtwistedness in this case. We need the following elementary

\begin{lemma}\label{lem:diff}
Let $(M,\xi=\ker(\alpha))$ be a contact manifold satisfying that $R_{\alpha}$ is complete. Denote the associated flow  $\phi^R_t$. Choose $f:\D^2_R\to\R$ a smooth function and $\lambda\in\Omega^1(\R^2)$ a primitive for $\omega_0=dx\wedge dy$. Define on $M\times\D^2_R$ the contact forms $\alpha_0=\alpha+\lambda$ and $\alpha_1=\alpha+\lambda+df$. Then the diffeomorphism

\begin{align*}
\Psi:M\times\D^2_R&\to M\times\D^2_R\\
(p,x,y)&\mapsto(\phi^R_{f(x,y)}(p),x,y)
\end{align*}
satisfies $\Psi^{*}\alpha_1=\alpha_0$.
\end{lemma}

The proof is left to the reader.

\begin{proof}[Proof of Proposition \ref{prop:OTInfty}]
Observe that there exists a smooth function $g:N\rightarrow \R$ such that $\tilde \alpha = e^g \alpha$ satisfies that $R_{\tilde \alpha }$ is complete\footnote{Standard fact, choose $g$ to be a ``reasonable" rapidly increasing proper function.}. Moreover we have the following diffeomorphism
\begin{eqnarray*}
\psi:N\times \R^2&\rightarrow& N\times\R^2\\
(p,r,\theta)&\mapsto&(p,e^{g/2}r,\theta),
\end{eqnarray*}
that clearly satisfies $\psi^\ast(\tilde\alpha+r^2d\theta)=e^g (\alpha + r^2 d\theta)$. Therefore, we can assume without loss of generality that $\alpha$ has a complete Reeb vector field.

It is sufficient to show that for any $K>0$, the manifold $ W = (N \times \mathbb{R}^2)\setminus (N \times B_{K}(0,0))$ is overtwisted. Let us prove it.

First, we realize that, since $(N, \ker(\alpha_{ot}))$ is overtwisted, there exists a positive constant $R=R(n)$ such that $(N \times B^2_{R}((0,0)),\ker(\alpha_{ot}+\lambda_{std}))$ is overtwisted \cite[Thm. 3.1]{CMP}. Now, let us consider the manifold $(N \times B^2_{R}((0,K+3R)),\ker(\alpha_{ot}+\lambda_{std}))$ embedded in $W$. We apply Lemma \ref{lem:diff} to show that both manifolds are contactomorphic.  Hence, $(N \times B^2_{R}((0,K+3R)),\ker(\alpha_{ot}+r^2d\theta))$ is overtwisted and thus, $N \times \mathbb{R}^2$  is overtwisted at infinity.

\end{proof}

Equivalently to the overtwisted case, there also exists a subclass of Legendrian embeddings, referred to as loose, which satisfies an $h$--principle type result \cite{Mu}. Let us define this class.

A \textit{formal Legendrian submanifold} $L$ of $M$ is an embedding $\phi:L\to M$ together with a family $\Phi_t:TL\to \phi^{*}TM$ such that $\Phi_t$ is a monomorphism for all $t\in[0,1]$ satisfying that $\Phi_0=d\phi$ and $\Phi_1(TL)\subset \phi^{*}\xi\subset \phi^{*}TM$. Notice that a Legendrian submanifold can be thought of as a formal Legendrian submanifold by letting $\Phi_s=d\phi$ for all $s$. In particular, two Legendrian embeddings $\phi_0$ and $\phi_1$ are formally isotopic if there exists a smooth isotopy $\phi_t$ between them and a homotopy of monomorphisms $\Phi_{t,s}:TL\to\phi_t^{*}TM$ such that $\Phi_{t,0}=d\phi_t$, $\Phi_{0,s}=d\phi_0$, $\Phi_{1,s}=d\phi_1$ and $\Phi_{t,1}(TL)\subset\Phi_t^{*}\xi$.

E. Murphy \cite{Mu} introduced the notion of loose Legendrian submanifolds. They are characterized by the following local model:

Consider an open ball $\B$ around the origin in $(\R^{3}, \xi_{std})$ where $\xi_{std}$ is the standard contact structure on $\R^{3}$ and let $L_0\subset(\R^{3}, \xi_{std})$ be a stabilized Legendrian arc as seen in Figure \ref{Fig:Loose} . Consider the zero section $\Gamma\subset T^{*}M$ of a closed manifold $M$ and denote by $U_{\Gamma}\subset T^{*}M$ an open neighborhood of it. Then, $(L_0\times\Gamma\subset(\B\times U_{\Gamma},\ker(\alpha_{std}+\lambda_{std}))$ is a Legendrian submanifold.

\begin{figure}[h!]
\centering
\includegraphics[scale=0.3]{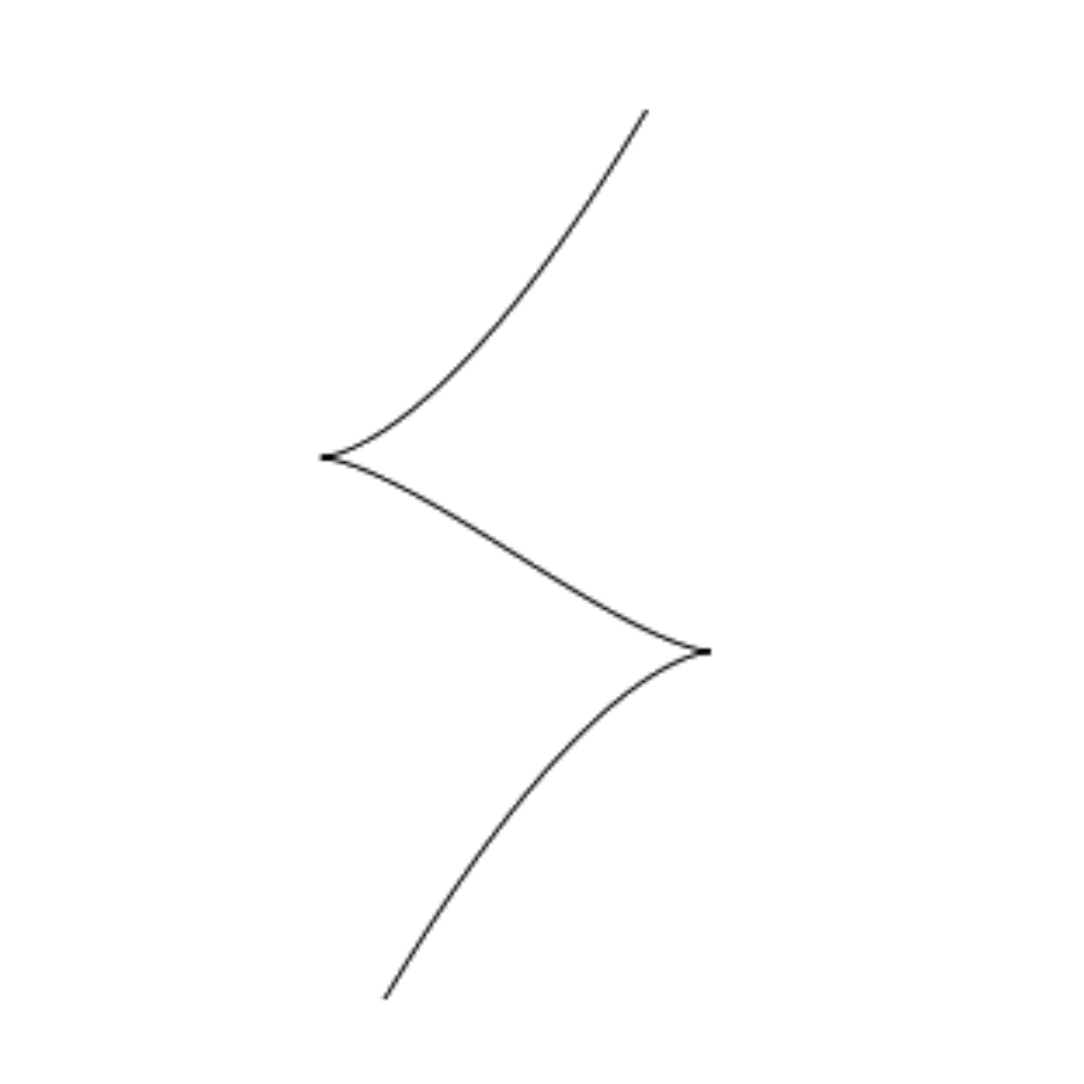}
\caption{The front projection of a stabilized Legendrian arc.}
\label{Fig:Loose}
\end{figure}

\begin{definition}\label{def:loose_chart}
The pair $(L_0\times\Gamma, \B\times U_{\Gamma})$ together with the contact structure $\ker(\alpha_{std}+\lambda_{std})$ is known as a \textit{loose chart}.
\end{definition}

\begin{definition}\label{def:loose}
A Legendrian submanifold $L^n\subset(M^{2n+1},\xi)$ with $n\geq2$ is called \textit{loose} if there exists an open set $U\subset M$ such that $((U, U\cap L), \xi)$ is contactomorphic to a loose chart.
\end{definition}

The corresponding $h$--principle can be stated as follows.

\begin{theorem}[\cite{Mu}]\label{lem:Loose_Surj}
Let $L^n\subset M^{2n+1}$ be a formal Legendrian submanifold with $n>1$. Then, there exists a loose Legendrian submanifold $\widetilde{L}$ such that they are formally isotopic. Moreover, given two formally isotopic loose Legendrians $L_1$ and $L_2$, they are isotopic through loose Legendrians.
\end{theorem}

\section{Proof of Theorem \ref{thm:product}}

Before proving Theorem \ref{thm:product}, we need to introduce a result due to Y. Eliashberg and L. Polterovich \cite{EP} adapted to the Legendrian case by V. Chernov and S. Nemirovski \cite{CN3} which states that if a Legendrian isotopy class contains a non-constant non-negative loop of Legendrians, then it contains a positive loop. More precisely,

\begin{lemma}[{\cite[Prop. 4.5]{CN3}}]\label{lem:main}
Let $\{\phi_t\}$ be a non-negative non-trivial Legendrian loop of closed Legendrians based at $L$. Then, there exists a positive loop of Legendrians $\{\phi_t^{'}\}$  which satisfies that $\phi_0(L) = \phi^{'}_0(L)$.

If we assume that $\phi_t$ is contractible then $\phi_t^{'}$ can be chosen to be contractible.
\end{lemma}

\begin{proof}
Given a smooth flow $\psi_t$ in $L$, we lift it to a contact flow $\widetilde{\psi}_t$ in $T^{*}L\times\R$ which preserves the zero-section with associated Hamiltonians $\widetilde{H}_t$. Then, choosing an appropriate cut-off function, we construct a family of contactomorphisms $\widehat{\psi}_t$ with support arbitrary close to the zero-section. Moreover, $\widehat {\psi}_t$ coincides with $\psi_t$ when restricted to $L$.

Now, let $G_t$ be the associated Hamiltonian for an extension $\varphi_t$ of the Legendrian loop $\phi_t$. Recall that $G_t\geq 0$. We can assume that there exists a point $p$ in the Legendrian and a time $t_0$ such that $G_{t_0}(p)>0$. Hence, there exists a neighborhood $U$ of $p\in L$ such that $G_{t_0}(q)>0$, for all $q\in U$. As $L$ is compact and the smooth flows of vector fields act transitively on $L$, there exists a finite set of flows $f_t^i$ such that the open sets $U_1= f_{0}^1(U), \cdots, U_n = f_{0}^n(U)$ cover $L$. Applying the construction above to $f_{t}^1, \cdots,f_{t}^n$, we get a family of contactomorphisms $\widehat{f}_t^1,\cdots,\widehat{f}_t^n$.

The loop $\phi_t^j= \widehat{f}_1^j \circ \phi_t$ with extension $ \varphi_t^{j} = \widehat{f}_1^j \circ \phi_t \circ (\widehat{f}_1^{j})^{-1}$ is positive in $U_j$ at $t_0$. Therefore $\Phi_t =(\varphi_t^1 \circ \cdots \varphi_t^n)$ is an extension of a non--negative loop of Legendrians based at $\phi_0(L)$ that is strictly positive for $t=t_0$. Now, fix $k$ big enough such that
$H(\Phi_t)$ is positive for $t\in[t_0,t_0+\frac{2}{2k}]$. Consider a finite open covering $(t_0, t_0+\frac{2}{2k}), (t_0+\frac{1}{2k}, t_0+\frac{3}{2k}), \cdots,(t_0-\frac{1}{2k},t_0+\frac{1}{2k})$ of $\SSS^1$. Then the conjugated loop $(\Phi_{-s})^{-1}\circ(\Phi_{t-s})_{|\Phi_0(L)}$ with extension $(\Phi_{-s})^{-1}\circ\Phi_{t-s} \circ \Phi_{-s}$ is positive in the interval $(t_0+s, t_0+s+\frac{2}{2k})$ and is based at $L$. Hence, the composition of this loop for $s=0, \frac{1}{2k}, \frac{2}{2k},\cdots, \frac{2k-1}{2k}$ is a positive loop based at $L$.

The proof follows with no changes in the contractible case.
\end{proof}

Theorem \ref{thm:product} is a consequence of the above Lemma.

\begin{proof}[Proof of Theorem \ref{thm:product}]
The contact vector field $X=\frac{\partial}{\partial \theta}$ generates a non--negative loop of contactomorphisms, that is positive away from $M \times \{0\}$.

Since $L$ is a Legendrian submanifold in $M \times D^2$ for dimensional reasons there exists a point of $L$ which is not in the contact submanifold $M \times \{0\}.$ Hence the loop restricted to the Legendrian is a non--negative non--trivial loop of Legendrians. Now we can apply Lemma~\ref{lem:main} to complete the proof.
\end{proof}

\begin{corollary}
Any Legendrian submanifold in $\mathbb{R}^{2n+1}$ admits a positive loop of Legendrians.
\end{corollary}

\begin{proof}
Since the standard contact manifold $\mathbb{R}^{2n+1}$ is nothing but $\mathbb{R}^{2n-1} \times \mathbb{R}^2$ with the contact structure given by $\alpha_{std} + r^2 d \theta$ where $\alpha_{std}$ is the standard contact form on $\mathbb{R}^{2n-1}.$ The corollary follows from Theorem~\ref{thm:product}.
\end{proof}

\begin{remark}
Being more careful, it can be shown that $\R^{2n+1}$ admits a positive loop of contactomorphisms. This is even true for $M\times\R^2$ just by using Lemma \ref{lem:diff}.
\end{remark}

\section{Proof of Theorem \ref{thm:main}}

The main idea of the proof is to construct a neighborhood $U_L$ of $L$ contactomorphic to $N\times\R^2$ satisfying the hypothesis of Lemma \ref{lem:swindling}. Then the result will follow from Theorem \ref{thm:product}. This is the content of Lemma \ref{lem:Nbhd}.

%We will use again Theorem \ref{thm:product} and Lemma \ref{lem:Nbhd}. Hence, we have to find a neighborhood $U_L$ of $L$ overtwisted at infinity and diffeomorphic to $N\times \R^2$ in such a way that the ambient contact structure $\xi$ restricted to $U_L$ must be formally equivalent to $\widetilde{\xi}=\ker(\alpha_{ot}+r^2 d\theta)$. This is the content of Lemma \ref{lem:Nbhd}.

\begin{lemma}\label{lem:Nbhd}
For any Legendrian submanifold $L\subset(M,\xi)$ satisfying the hypothesis of Theorem \ref{thm:main}, there exists a neighborhood $U_L$ of $L$ diffeomorphic to $N\times\R^2$ such that $(U_L, \xi)$ is overtwisted at infinity and $N$ is an open manifold if $n\geq2$.
\end{lemma}

\begin{proof}
By the first hypothesis, a small tubular neighborhood $V_L$ of $L$ is diffeomorphic to $N\times\R^2$. By the second hypothesis, there exists an overtwisted disk which does not intersect $L$. Therefore, $V_L$ can be chosen disjoint from the overtwisted embedding $U_{ot}$.  Also,  according to  \cite{CMP}, the overtwisted embedding $U_{ot}$ is overtwisted at infinity. Therefore, $U_L$ will be the connected sum of $V_L$ with $U_{ot}$ along a tubular neighborhood of a path connecting their boundaries (see Figure \ref{Fig:Thm_U_OT}).

\begin{figure}[h!]
\centering
\includegraphics[scale=1.3]{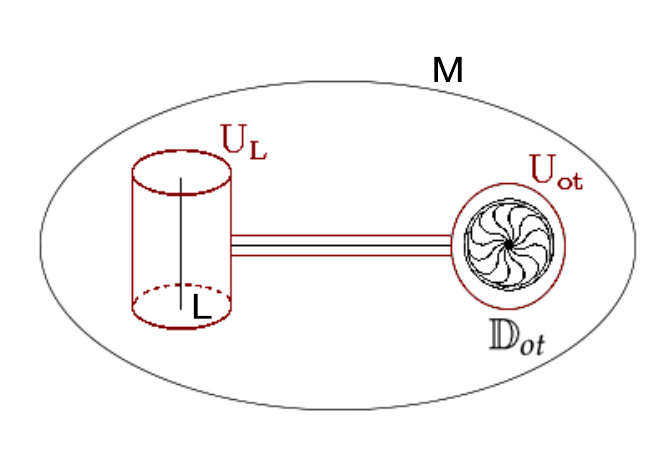}
\caption{Construction of $U_L$.}
\label{Fig:Thm_U_OT}
\end{figure}

Hence $U_L$ is overtwisted at infinity by construction and is diffeomorphic to $\N\times\R^2$.
\end{proof}

In order to apply Lemma \ref{lem:swindling} we use Proposition \ref{prop:OTInfty} to find an overtwisted at infinity contact manifold of type $N\times\R^2$. Hence, %since if $N$ has dimension $1$ it does not hold,
we have to distinguish two cases.

\subsection{Proof of Theorem \ref{thm:main} for $n>1$.}

It follows from Lemma \ref{lem:Nbhd} that there exists a diffeomorphism $\Phi:U_L\to N \times \mathbb{R}^2$. In addition, $(U_L,\xi)$ is overtwisted at infinity. By Lemma \ref{lem:top_1} the submanifold $N\times\{0\}$ can be equipped with an almost contact structure $(\xi_N, J_N)$ such that $(\xi_N\oplus\R^2,J_N\oplus i)$ represents the same formal contact class as $\xi$.

By Theorem \ref{thm:BEM} there exists an overtwisted contact structure $\xi_{ot}= \ker( \alpha_{ot})$ on $N$ formally homotopic to $\xi_N$. Therefore, the contact structure $\xi'=\ker(\alpha_{ot}+r^2d\theta)$ is overtwisted at infinity by Proposition \ref{prop:OTInfty} and formally homotopic to $\xi$. By Lemma \ref{lem:swindling}, there is a diffeomorphism $F:U_L\to N\times\R^2$ taking $\xi$ to $\xi'$.

By Theorem \ref{thm:product}, $F(L)$ admits a positive loop of Legendrians $\phi_t$. Thus, the family $\phi_t\circ F^{-1}$ is a positive loop of Legendrians for $L$.

\subsection{Proof of Theorem \ref{thm:main} for $n=1$.}\label{sec:3dim}

Let $L\hookrightarrow (M,\xi)$ denote the Legendrian embedding. A tubular neighborhood can be identified with $L\times\D_{\varepsilon}^2\subset(M,\xi)$. By hypothesis, there exists $\varepsilon>0$ such that there is an overtwisted disk which does not intersect $L\times\D_{\varepsilon}^2$. Therefore, we can find a neighborhood $U_L$ of $L$ diffeomorphic to $\SSS^1\times\R^2$ such that $(U_L, \xi)$ is overtwisted at infinity.

Consider now the contact manifold $(\SSS^1(z) \times\R^2(r,\theta), \eta=\ker(dz+r^2d\theta))$. Here, $\partial_z$ is a positive loop with Hamiltonian $H=1$, in particular it is autonomous. Fix a sequence of transverse knots $\gamma_k=(z, \varepsilon(1-1/k), 0)$, with $k\in \mathbb{N}^*$ Then, the contact manifold obtained as a sequence of half Lutz twists along each of them is overtwisted at infinity. It admits a positive loop by \cite{CP}. Denote it by $(\SSS^1\times\D_{\varepsilon}^2, \eta^{\gamma})$.

Finally, $\xi$ and $\eta^{\gamma}$ are formally equivalent because there exists only one class of formal contact structures on $\SSS^1\times\D_{\varepsilon}^2$. Again, the claim follows by using Lemma \ref{lem:swindling}.

\section{Proof of Theorem \ref{thm:loose}}

We will use again Theorem \ref{thm:product}. Hence we need to construct a neighborhood of $L$ contactomorphic to  $(N\times\D^2,\ker(\alpha_N+r^2d\theta)$ with some contact manifold $(N,\alpha_N)$.

We first prove a simple case.

\subsection{Euler characteristic zero}
Assume that $T^{*}L$ has a never--vanishing section. Using Weinstein's tubular neighborhood theorem, we find a neighborhood $U_L$ of $(L,\ker(\alpha))$ contactomorphic to $(T^{*}L\times\R(z),\ker(dz-\lambda_{std}))$. As  $(T^{*}L\setminus\{0\}, d\lambda_{std})\text{ and }(\SSS(T^{*}L)\times\R, d(e^t\lambda_{std}))$ admit a diffeomorphism preserving the Liouville forms, the natural inclusion $\SSS(T^{*}L) \hookrightarrow T^{*}L\times\R$ is a contact embedding. By the tubular neighborhood theorem, there exists a neighborhood $V$ of $\SSS T^{*}L$ contactomorphic to $\SSS T^{*}L\times\D^2_{\varepsilon}$.

%Hence, by the tubular neighborhood theorem, $(\SSS(T^{*}L)\times\D^2,\lambda_{std}+r^2d\theta)$ is a contact manifold.

%Now, we can push $L$ so close to $\SSS(T^{*}L)$ as we want by an isotopy. Denote by $L_1$ the image of $L$ by this isotopy and by $\alpha_1$, the new form on $\SSS(T^{*}L)$. Notice that now we cannot assure $L_1$ to be Legendrian, but it will be formally Legendrian instead. The new form may not be contact any more, so we make a $\CC^{\infty}$--small perturbation to $\alpha_1$ in order to make it contact. Call $\alpha_2$ this new form and $L_2$ the image of $L_1$ by this perturbation which is still a formal Legendrian submanifold.  Notice that $L_2$ is still as close as we want to $\SSS(T^{*}L)$. Finally, using an $h$--principle result (see, e.g., {\cite[Thm. 12.3.1]{EM}}), $(L,\ker(\alpha))$ and $(L_2,\ker(\alpha_2))$ are isotopic. Hence, they are formally equivalent.

The never--vanishing section provides an embedding $\sigma: L \to \SSS T^{*}L\subset T^{*}L\times\R$. Thus, we obtain a family of embeddings $\sigma_t:L \to T^{*}L\times\R$ defined as $\sigma_t=t\sigma$. Since $\sigma_0$ is a Legendrian embedding, the whole family $\sigma_t$ can be lifted into a family $(\sigma_t,\Phi_t)$ of formal Legendrian embeddings.

Applying Theorem \ref{lem:Loose_Surj} to $(\sigma_1, \Phi_1)$ as a formal Legendrian embedding into $V$ we obtain a family $(\sigma_t, \Phi_t)$ with $t\in [1,2]$ of formal embeddings in $V$ satisfying that $\sigma_2$ is a loose Legendrian embedding. The family $(\sigma_t, \Phi_t)$ with $t\in[0,2]$ satisfies the hypothesis of Theorem \ref{lem:Loose_Surj} and so, it can be deformed relative to $t=0,2$ into a Legendrian isotopy inside $M$. We are, thus, reduced to find a positive loop for the loose Legendrian $\sigma_2$. But this is true by Theorem \ref{thm:product} applied to $V$.

\subsection{General case}

By hypothesis, we have that a neighborhood $U_L$ of $L$ is diffeomorphic to $N\times\R^2$, for an open manifold $N$. By Lemma \ref{lem:top_1}, we assume that there is a formal contact structure $(\xi_N, J_N)$ on $N$ such that $(\xi_N\oplus\R^2, J_N\oplus i)$ is the formal contact class of $\xi$. By Theorem {\cite[Thm. 10.3.2]{EM}}, the formal contact structure $\xi_N= \ker \alpha_N$ can be assumed to be contact.

We are in the hypothesis of {\cite[Thm. 12.3.1]{EM}}. Therefore, the formal contact embedding $e_0: N\hookrightarrow N\times\{ 0 \} \subset N\times\R^2\simeq U_L$ admits an isotopy of formal contact embeddings $e_t: N\to U_L$ satisfying that $e_1$ is a contact embedding. By the contact neighborhood theorem, there exists $\phi_1:N\times\D^2_{\varepsilon}\hookrightarrow U_L$, for sufficiently small $\varepsilon>0$, such that

\begin{enumerate}
  \item $(\phi_1)_{|N\times{0}}=e_1$.
  \item Fix the contact form $\alpha=\alpha_N+r^2d\theta$ in the manifold $N\times \D^2_{\varepsilon}$. The map $\phi_1$ is a contact embedding.
\end{enumerate}

%Denote by $U_0=\phi_0(N\times\D^2_{\varepsilon})$ that is an open neighborhood of $N\times\{0\}\subset N\times\R^2$. Define the contraction embedding
%
%\begin{eqnarray*}
%C_t:N\times\R^2 & \to & N\times\R^2\\
%(p,v)&\mapsto&(p,tv)
%\end{eqnarray*}

By construction we have $L\subset N$. Define the family of embeddings $\varphi_t:L\to U_L,\; t\in[0,1]$ as $\varphi_t=(e_t)_{|L}$.

Promote the family $\varphi_t$ into a family of formal Legendrian embeddings $(\varphi_t,\Phi_t), \; t\in[0,1]$. Apply Theorem \ref{lem:Loose_Surj} to $(\varphi_1,\Phi_1)$ as formal Legendrian embedding of the manifold $\phi_1(N\times \D^2_{\varepsilon})$, to create a family of formal Legendrians embeddings $(\varphi_t, \Phi_t)\; t\in[1,2]$ such that $(\varphi_2, \Phi_2)$ is a loose Legendrian embedding into $\phi_1(N\times \D^2_{\varepsilon})$. Since, by hypothesis $\varphi_0$ is loose, we can apply Theorem \ref{lem:Loose_Surj} to show that $\varphi_0$ and $\varphi_2$ are Legendrian isotopic in $M$.

But the image of $\varphi_2$ lies in $\phi_1(N\times\D^2_{\varepsilon})$. Thus, $(\phi_1)^{-1}\circ\varphi_2$ is a Legendrian embedding into $(N\times\D^2_{\varepsilon}, \ker(\alpha_N+r^2d\theta))$. Theorem \ref{thm:product} concludes the claim.

\section{Proof of Theorem \ref{thm:R4} and Corollaries}

\begin{proof}[Proof of Theorem \ref{thm:R4}]
Notice that $\mathrm{U}(2)$ acts by contactomorphisms on $M\times\D^4_{\varepsilon}$. %since any element of $\mathrm{U}(2)$ acting in $\R^4$ preserves $x_1dy_1-y_1dx_1+x_2dy_2-y_2dx_2$.
%Now consider the two following loops in $\mathrm{U}(2)$:
%
%\[\begin{aligned}
%X_1=\partial \theta_1& \text{ with associated Hamiltonian } H_1=r_1^2,\\
%X_2=\partial \theta_2& \text{ with associated Hamiltonian } H_2=r_2^2.
%\end{aligned}\]

Now consider the contact vector fields $X_1=\partial_{\theta_1}$ and $X_2=\partial_{\theta_2}$ with associated Hamiltonians $H_1=r_1^2$ and $H_2=r_2^2$, respectively. The contact vector field $X=X_2-X_1=\partial_{\theta_2}-\partial_{\theta_1}$, whose associated Hamiltonian is $H=r_2^2-r_1^2$, generates a loop  that preserves $M^{+}$  and is positive on this domain. Denote by $A_t$ the unitary matrix

\[\left(
\begin{array}{cc}
e^{2\pi it}& 0\\
0& e^{-2\pi it}
\end{array}
\right),\]

then the flow associated to $X$ reads as $\phi_t(p,\left(\begin{aligned}
z_1\\
z_2
\end{aligned}\right))=\left(p, A_t\left(\begin{aligned}
z_1\\
z_2
\end{aligned}\right)\right)$.

Realize that $A_t$ is contractible in $\mathrm{U}(2)$ since $det(A_t)=1$ and $\mathrm{S}\mathrm{U}(2)$ is simply connected. Therefore, there exists a family of loops $\widetilde{A}_{t,s}\in\mathrm{U}(2)$ with $s\in[0,1]$ such that

\[
\begin{array}{c}
\widetilde{A}_{t,0}=Id,\\
\widetilde{A}_{t,1}=A_t.
\end{array}
\]

Hence,
$\phi_{t,s}(p,\left(\begin{aligned}
z_1\\
z_2
\end{aligned}\right))=\left(p, \widetilde{A}_{t,s}\left(\begin{aligned}
z_1\\
z_2
\end{aligned}\right)\right)$
is the contraction of the positive loop.

\end{proof}

\begin{proof}[Proof of Corollary \ref{cor:1}]
We mimic the proof of Theorem \ref{thm:loose}. A neighborhood $U_L$ of $L$ is diffeomorphic to $N\times \R^4$. By application of classical $h$--principles, we can find an isotopy $\phi_t: N \times D_{\varepsilon}^4 \to U_L$ such that is the identity for $t=0$ and is a contact embedding for $t=1$.

Denote by $\varphi_0: L \to U_L$ the given Legendrian embedding. We create a path of formal Legendrians embeddings $(\varphi_t, \Phi_t)$ starting at $\varphi_0$ an such that $\varphi_1(L) \subset \phi_1(N^+) \subset \phi_1(N \times D_{\varepsilon}^4)$. Finally, applying twice Theorem \ref{lem:Loose_Surj} and Theorem \ref{thm:R4}, we conclude the result.
\end{proof}

\begin{proof}[Proof of Corollary \ref{cor:2}]
$(\R^{2n+1},\xi_{std})$ is contactomorphic to $(\R^{2n-3}\times\R^4,\ker{\alpha_{std}+r_1^2d\theta_1+r_2^2d\theta_2})$. By compactness of the Legendrian submanifold, we can assume that $L \subset \R^{2n-3}\times\D^2_R\times\D^2_R$, for some $R>0$.

By a simple refinement of Lemma \ref{lem:diff}, the domains $\R^{2n-3}\times\D^2_R \times \D^2_R$ and $\R^{2n-3}\times\D^2_R(0,0) \times \D^2_R(10R,0)$ are contact isotopic. Therefore, we can assume that the Legendrian embedding lies in $\R^{2n-3} \times \D^2_R(0,0) \times \D^2_R(10R,0) \subset (\R^{2n-3})^+$. We apply Theorem \ref{thm:R4}.
\end{proof}

\begin{proof}[Proof of Corollary \ref{cor:3}]
Consider $(\R^{2n-3},\widetilde{\xi_{ot}}=\ker(\widetilde{\alpha_{ot}}))$ with $\widetilde{\xi_{ot}}$ any overtwisted contact structure on $\R^{2n-3}$.  $(\R^{2n-3}\times\R^4,\ker(\widetilde{\alpha_{ot}}+r_1^2d\theta_1+r_2^2d\theta_2))$ is the overtwisted at infinity contact structure on $\R^{2n+1}$. The complementary of $L$ is overtwisted, then $L$ is loose. We use Theorem \ref{lem:Loose_Surj} to create a contact isotopy of $L$ into $(\R^{2n-3})^+$. We apply Theorem \ref{thm:R4}.
\end{proof}

\end{document}